\documentclass[11 pt]{amsart}

\usepackage{amsmath}             
\usepackage{amssymb}             
\usepackage{latexsym}            
\usepackage[active]{srcltx}
\usepackage{ae}
\numberwithin{figure}{section}

\numberwithin{figure}{section}

\newtheorem{theorem}{Theorem}[section]
\newtheorem{lemma}[theorem]{Lemma}

\newtheorem{corollary}[theorem]{Corollary}
\theoremstyle{definition}
\newtheorem{definition}[theorem]{Definition}

\newtheorem{remark}[theorem]{Remark}
\numberwithin{equation}{section}

\newcommand{\norm}[1]{\left|\left|#1\right|\right|}

\def\vint_#1{\mathchoice%
          {\mathop{\kern 0.2em\vrule width 0.6em height 0.69678ex depth -0.58065ex
                  \kern -0.8em \intop}\nolimits_{\kern -0.4em#1}}%
          {\mathop{\kern 0.1em\vrule width 0.5em height 0.69678ex depth -0.60387ex
                  \kern -0.6em \intop}\nolimits_{#1}}%
          {\mathop{\kern 0.1em\vrule width 0.5em height 0.69678ex depth -0.60387ex
                  \kern -0.6em \intop}\nolimits_{#1}}%
          {\mathop{\kern 0.1em\vrule width 0.5em height 0.69678ex depth -0.60387ex
                  \kern -0.6em \intop}\nolimits_{#1}}}
\def\vintslides_#1{\mathchoice%
          {\mathop{\kern 0.1em\vrule width 0.5em height 0.697ex depth -0.581ex
                  \kern -0.6em \intop}\nolimits_{\kern -0.4em#1}}%
          {\mathop{\kern 0.1em\vrule width 0.3em height 0.697ex depth -0.604ex
                  \kern -0.4em \intop}\nolimits_{#1}}%
          {\mathop{\kern 0.1em\vrule width 0.3em height 0.697ex depth -0.604ex
                  \kern -0.4em \intop}\nolimits_{#1}}%
          {\mathop{\kern 0.1em\vrule width 0.3em height 0.697ex depth -0.604ex
                  \kern -0.4em \intop}\nolimits_{#1}}}

\newcommand{\kint}{\vint}

\newcommand{\aveint}[2]{\mathchoice%
          {\mathop{\kern 0.2em\vrule width 0.6em height 0.69678ex depth -0.58065ex
                  \kern -0.8em \intop}\nolimits_{\kern -0.45em#1}^{#2}}%
          {\mathop{\kern 0.1em\vrule width 0.5em height 0.69678ex depth -0.60387ex
                  \kern -0.6em \intop}\nolimits_{#1}^{#2}}%
          {\mathop{\kern 0.1em\vrule width 0.5em height 0.69678ex depth -0.60387ex
                  \kern -0.6em \intop}\nolimits_{#1}^{#2}}%
          {\mathop{\kern 0.1em\vrule width 0.5em height 0.69678ex depth -0.60387ex
                  \kern -0.6em \intop}\nolimits_{#1}^{#2}}}

\newcommand{\eps}{\varepsilon}
\newcommand{\R}{\mathbb{R}}

\renewcommand{\limsup}{\operatornamewithlimits{lim \, sup}}
\renewcommand{\liminf}{\operatornamewithlimits{lim \, inf}}
\newcommand{\esssup}{\operatornamewithlimits{ess\, sup}}

\newcommand{\spt}{\operatorname{spt}}

\renewcommand{\l}{\left}
\renewcommand{\r}{\right}
\newcommand{\parts}[2]{\frac{\partial {#1}}{\partial {#2}}}

\numberwithin{equation}{section}

\newcommand{\trm}{\textrm}
\newcommand{\half}{{\frac{1}{2}}}
\newcommand{\q}[1]{Q_{#1}}

\newcommand{\Om}{\Omega}

\newcommand{\limminus}{{\mathchoice{\raise.17ex\hbox{$\scriptstyle -$}}
                {\raise.17ex\hbox{$\scriptstyle -$}}
                {\raise.1ex\hbox{$\scriptscriptstyle -$}}
                {\scriptscriptstyle -}}}
\newcommand{\limplus}{{\mathchoice{\raise.17ex\hbox{$\scriptstyle +$}}
                {\raise.17ex\hbox{$\scriptstyle +$}}
                {\raise.1ex\hbox{$\scriptscriptstyle +$}}
                {\scriptscriptstyle +}}}

\newcommand{\abs}[1]{\left| #1 \right|}
\newcommand{\Lip}{\operatorname{Lip}}
\newcommand{\beq}{\begin{equation}}
\newcommand{\eeq}{\end{equation}}

\title[Local higher integrability of parabolic quasiminimizers]{Local higher integrability for parabolic quasiminimizers in metric spaces}

\author{M. Masson, M. Miranda Jr., F. Paronetto, and M. Parviainen}

\address{Mathias Masson, Aalto University,
Department of Mathematics and Systems Analysis,
P.O. Box 11100
FI-00076 Aalto, Finland}
\email{mathias.masson.finland@gmail.com}

\address{Michele Miranda Jr., Department of Mathematics and Computer Science, University of Ferrara via Machiavelli 35, 44121 Ferrara, Italy}
\email{michele.miranda@unife.it}

\address{Fabio Paronetto,
Dipartimento di Matematica, Universit`a degli Studi di Padova via Trieste 63, 35121 Padova, Italy
}
\email{fabio.paronetto@unipd.it}

\address{Mikko Parviainen,
Department of Mathematics and Statistics, P.O. Box 35 (MaD), FI-40014 University of Jyv\"askyl\"a, Finland
}
\email{mikko.j.parviainen@jyu.fi}

\thanks{}

\begin{document}


\begin{abstract}
Using variational methods, we prove local higher integrability for the minimal p-weak upper gradients of parabolic quasiminimizers  in metric measure spaces. We assume the measure to be doubling and the underlying space to be such that a weak $(1,2)$-Poincar\'e inequality is supported.

\end{abstract}
\maketitle

\section{Introduction}

In the Euclidean setting, finding a solution to the $p$-parabolic partial differential equation in $\Omega\times(0,T)$,
\begin{align}\label{p-parabolic equation}
-\frac{\partial u}{\partial t}+\textrm{div}(|\nabla u |^{p-2} \nabla u )=0,
\end{align}
can be formulated into an equivalent variational problem, which is to find a function $u$, such that  with $K=1$ we have
\begin{equation}\label{variational form}
\begin{split}
p\int_{\{\phi\neq 0\}} u \frac{\partial \phi}{\partial t} \,dx\,dt+ &\int_{\{\phi \neq 0\}} |\nabla u|^p\,dx\,dt \\
&\leq K \int_{\{\phi \neq 0\}}|\nabla u+\nabla\phi|^p\,dx\,dt,
\end{split}
\end{equation}
for every compactly supported $\phi \in C_0^{\infty}(\Omega\times (0,T))$. Here $\Omega\subset \R^d$ denotes a domain. A generalization of this variational problem is to consider \eqref{variational form} with the weakened  assumption $K\geq 1$. The function $u$ is then called a parabolic $K$-quasiminimizer related to \eqref{p-parabolic equation}. When $p=2$, equation \eqref{p-parabolic equation} is the classical heat equation.

Our main result is to show that a parabolic quasiminimizer $u$ in a doubling metric measure space has the following higher integrability property: The upper gradient \cite{HeinKosk98} of $u$ is locally integrable to a slightly higher power than is initially assumed (Theorem \ref{local higher integrability}). Although being local, the estimate we obtain is scale and place invariant inside the set.

As a first step for examining higher regularity of quasiminimizers in metric measure spaces, we only treat the simplest case $p=2$, and so the quasiminimizers in this paper are related to the heat equation. Assuming a weak $(1,2)$-Poincaré inequality and a doubling measure, starting from the definition of quasiminimizers we prove an energy type estimate and a Caccioppoli type inequality for $u$. Then using these and Sobolev--Poincaré's inequality we show a reverse Hölder inequality, from which higher integrability follows.

 The novelty of this paper is that we prove our results in the general metric measure space setting, using a purely variational approach. No reference is made to the explicit scaling properties of the measure, or on assumptions that the measure is translation invariant. Instead we rely on taking integral averages and on the assumption that the measure is doubling. Also, no reference is made to the equation, as the local notion of a weak solution is replaced by quasiminimizers, which are in general known to not be uniquely determined by the quasiminimizing property. The variational nature of quasiminimizers opens up the possibility to substitute gradients by upper gradients, which do not require the existence of partial derivatives on the underlying space. This way, the theory of parabolic PDEs is extended to metric measure spaces.
 
 Methods established in the general metric space setting can be expected to be robust and of fundamental nature, in the sense that they are largely independent of the geometry of the underlying measure space. In connection with this, it is worth noting that already in the Euclidean case, an open question is to show that a weak solution to the doubly non-linear partial differential equations is locally higher integrable. For $p=2$, the $K$-quasiminimizer related to the doubly non-linear partial differential equation coincides with the one used in this paper.

In the elliptic case quasiminimizers have been extensively investigated. Originally, quasiminimizers were introduced in the elliptic setting by Giaquinta and Giusti \cite{GiaqGius82, GiaqGius84} as a tool for a unified treatment of variational integrals, elliptic equations and systems, and quasiregular mappings in $\R^d$. Giaquinta and Giusti realized that De Giorgi's method \cite{DeGi57} could be extended to quasiminimizers, and they showed, in particular, that elliptic quasiminimizers are locally Hölder continuous. DiBenedetto and Trudinger \cite{DiBeTrud84} proved the Harnack inequality for quasiminimizers. As mentioned, unlike partial differential equations, quasiminimizers are a purely variational notion, and so Kinnunen and Shanmugalingam \cite{KinnShan01} were able to extend these regularity results for elliptic quasiminimizers into the general metric setting by using upper gradients.

In comparison to the elliptic case, already in the Euclidean case the literature available for parabolic quasiminimizers is relatively limited. Following Giaquinta and Giusti, Wieser \cite{Wies87} generalized the notion of quasiminimizers to the parabolic setting in Euclidean spaces. Parabolic quasiminimizers have also been studied by Zhou \cite{Zhou93,Zhou94}, Gianazza and Vespri \cite{GianVesp06c}, Marchi \cite{Marc94} and Wang \cite{Wang88}. Recently, also the notions of parabolic quasiminimizers has  been extended and studied in metric spaces \cite{KinnMaroMiraParo12},\cite{MassSilj11}.

 Higher integrability results were introduced in the parabolic setting by Giaquinta and Struwe \cite{GiaqStru82}, when they proved reverse Hölder inequalities and local higher integrability in the case $p=2$, for weak solutions of parabolic second order $p$-growth systems. Kinnunen and Lewis \cite{KinnLewi00} extended this local result to the general degenerate and singular case $p\neq 2$. Recently, several authors have worked in the parabolic setting on questions concerning local and global higher integrability and reverse Hölder inequalities, see \cite{Misa06}, \cite{AcerMing07}, \cite{Boge08, Boge09}, \cite{Parv09a,Parv09b}, \cite{BogeParv10}, \cite{BogeDuzaMing11},\cite{Fuga12}, and in particular for quasiminimizers see \cite{Parv08} and \cite{Habe}.  


\section{Preliminaries}

\subsection{Doubling measure}

Let
$X=(X,d,\mu)$ be a
complete metric space endowed with a metric $d$ and a positive complete
doubling Borel measure $\mu $ which supports a weak $(1,2)$-Poincaré inequality. 

The measure $\mu$ is called \emph{doubling} if there exists a constant
$c_\mu \geq 1$, such that for all balls $B= B(x_0,r):=\{x\in X:
d(x,x_0)<r\}$ in~$X$,
\begin{equation*}
        \mu(2B) \le c_\mu \mu(B),
\end{equation*}
where $\lambda B=B(x_0,\lambda r)$.


 By iterating the doubling condition, it follows with
$s=\log_{2}c_{\mu}$ and $C=c_\mu^{-2}$ that
\begin{equation} \label{eq:DoublingConsequence}
\frac{\mu(B(z,r))}{\mu(B(y,R))} \geq C\Bigl(\frac{r}{R}\Bigr)^{s},
\end{equation}
for all balls $B(y,R)\subset X$, $z \in B(y,R)$
and $0 < r \leq R < \infty$.
However, here $c_\mu$ does not have to be optimal. From now on we fix $c_\mu>1$ and so $s>0$.

\subsection{Notation}

Next we introduce more notation  used throughout this paper. 
 Given  any $z_0=(x_0,t_0)\in X\times \R \trm{ and
}{\rho}>0,$ let
\[
B_\rho(x_0)=\{\,x\in X\, :\, d(x,x_0)<\rho\,\},
\]
denote an open ball in $X$, and let
\[
\Lambda_{\rho}(t_0)=(t_0-\half {\rho}^2,\, t_0+\half {\rho}^2),
\]
denote an open interval in $\R$. A space-time cylinder in
$X\times \R$ is denoted by
\[
\q{{\rho}}(z_0)=B_{\rho}(x_0)\times \Lambda_{\rho}(t_0),
\]
so that $\nu(\q{{\rho}}(z_0))=\mu(B_{\rho}(x_0))\rho^2$.
When no confusion arises, we shall omit the reference points
  and write briefly $B_{\rho},\ \Lambda_{\rho}$ and $Q_{\rho}$.
 We denote the product measure by $
d \nu=d \mu\, d t$. The integral average of $u$ is
denoted by
\begin{align}\label{spatial mean value}
u_{B_\rho}(t)=\kint_{B_{\rho}} u(x,t)\, d \mu=\frac{1}{\mu({B_{\rho})}}\int_{B_{\rho}} u(x,t)\, d \mu
\end{align}
and
\[
\begin{split}
\kint_{Q_\rho} u \,d \nu =\frac{1}{\nu(Q_\rho)}\int_{Q_\rho} u \,d \nu.
\end{split}
\]
Let $\Omega\subset X$ be a domain, and let $0<T<\infty$. We denote $\Omega_T=\Omega\times (0,T)$.

%
\subsection{Upper gradients.} Following \cite{HeinKosk98}, a non-negative Borel measurable function $g: \Omega \rightarrow [0, \infty]$ is said to be an upper gradient of a function $u: \Omega\rightarrow [-\infty, \infty]$ in $\Omega$, if for all compact rectifiable paths $\gamma$ joining $x$ and $y$ in $\Omega$ we have 
\begin{align}\label{upper gradient}
|u(x)-u(y)|\leq \int_{\gamma} g \,ds.
\end{align}
In case $u(x)=u(y)=\infty$ or $u(x)=u(y)=-\infty$, the left side is defined to be $\infty$. Assume $1\leq p <\infty$. The $p$-modulus of a family of paths $\Gamma$ in $\Omega$ is defined to be
\begin{align*}
 \inf_\rho \int_\Omega \rho^p \,d\mu,
\end{align*}
where the infimum is taken over all non-negative Borel measurable functions $\rho$ such that for all rectifiable paths $\gamma$ which belong to $\Gamma$, we have
\begin{align*}
 \int_{\gamma} \rho \, ds \geq 1.
\end{align*}
A property is said to hold for $p$-almost all paths, if the set of non-constant paths for which the property fails is of zero $p$-modulus. Following \cite{KoskMacm98,Shan00}, if \eqref{upper gradient} holds for $p$-almost all paths $\gamma$ in $\Omega$, then $g$ is said to be a $p$-weak upper gradient of $u$. 

When $1\leq p<\infty$ and $u\in L^p(\Omega)$, it can be shown \cite{Shan01,Hajl} that  there exists a minimal $p$-weak upper gradient of $u$, we denote it by $g_u$, in the sense that $g_u$ is a $p$-weak upper gradient of $u$
and for every $p$-weak upper gradient $g$ of $u$ it holds $g_u\leq g$ $\mu$-almost everywhere in $\Omega$. 
Moreover, if $v=u$ $\mu$-almost everywhere in a Borel set $A\subset \Omega$, then $g_v=g_u$ $\mu$-almost everywhere in $A$. Also, if $u,v\in L^p(\Omega)$, then $\mu$-almost everywhere in $\Omega$, we have
\begin{align*}
&g_{u+v}\leq g_u+g_v,\\
&g_{uv}\leq |u|g_v+|v|g_u.
\end{align*}
Proofs for these properties and more on upper gradients in metric spaces can be found for example in \cite{BjorBjor11} and the references therein. See also \cite{Chee99} for a discussion on upper gradients.

\subsection{Newtonian spaces.} Following \cite{Shan00}, for  $1<p<\infty$, and $u\in L^p(\Omega)$, we define
\begin{align*}
\|u\|_{1,p,\Omega}=\|u\|_{L^p(\Omega)}+\|g_u\|_{L^p(\Omega)},
\end{align*}
and
\begin{align*}
\widetilde N^{1,p}(\Omega)= \{ u\,:\, \|u\|_{1,p,\Omega}<\infty\}.
\end{align*}
An equivalence relation in $\widetilde N^{1,p}(\Omega)$ is defined by saying that $u\sim v$ if
\begin{align*}
 \|u-v\|_{\widetilde N^{1,p}(\Omega)}=0.
\end{align*}
The \emph{Newtonian space} $N^{1,p}(\Omega)$ is defined to be the space $\widetilde N^{1,p}(\Omega)/ \sim$, with the norm
\begin{align*}
 \|u\|_{N^{1,p}(\Omega)}=\|u\|_{1,p,\Omega}.
\end{align*}

A function $u$ belongs to the local Newtonian space $N_{\textrm{loc}}^{1,p}(\Omega)$ if it belongs to $N^{1,p}(\Omega')$ for every $\Omega' \subset \subset \Omega$. The Newtonian space with zero boundary values is defined as 
\begin{align*}
N_0^{1,p}(\Omega)=\{\, u|_\Omega\,:\,u \in N^{1,p}(X),\,u=0 \textrm{ in }X\setminus \Omega\,\}.
\end{align*}
In practice this means that a function belongs to $N_0^{1,p}(\Omega)$ if and only its zero extension to $X\setminus \Omega$ belongs to $N^{1,p}(X)$. For more properties of Newtonian spaces, see \cite{Hein01,Shan00, BjorBjor11}.

\subsection{Poincar\'e's and Sobolev's inequality}
For $1\leq q <\infty$, $1< p < \infty$, the measure $\mu$ is said to support a weak $(q,p)$-Poincar\'e inequality if there exist constants $c_P>0$ and $\lambda\ge 1$ such that
\begin{equation}
\l(\vint_{B_\rho(x)}|v-v_{B_\rho(x)}|^q \, d\mu\r)^{1/q} \le c_P \rho\left(\vint_{ B_{\lambda\rho}(x)} g_v^p \, d\mu\right)^{1/p},
\end{equation}
for every $v\in N^{1,p}(X)$ and $B_\rho(x)\subset X$. In case $\lambda=1$, we say a $(q,p)$-Poincaré inequality is in force. In a general metric measure space setting, it is of interest to have assumptions which are invariant under bi-Lipschitz mappings. The weak $(q,p)$-Poincaré inequality has this quality. 

For a metric space $X$ equipped with a doubling measure $\mu$, it is a result by Hajlasz and Koskela \cite{HajlKosk95} that the following Sobolev inequality holds: If  $X$ supports a weak $(1,p)$-Poincaré inequality for some $1<p<\infty$, then $X$ also supports a weak $(\kappa,p)$-Poincaré
inequality, where
\begin{equation*}
\kappa=\begin{cases}\frac{s p}{s-p},  &\text{for} \quad 1<p<s, \\ 2p, &\text{otherwise}, \end{cases}
\end{equation*}
possibly with different constants $c_P'>0$ and $\lambda'\geq 1$. 

\begin{remark} \label{poincare_remark}
It is a recent result by Keith and Zhong \cite{KeitZhon08}, that when $1<p<\infty$ and $(X,d)$ is a complete metric space with doubling measure $\mu$, the weak $(1,p)$-Poincaré inequality implies a weak $(1,q)$-Poincar\'e inequality for some $1<q<p$. Then by the above discussion, $X$ also supports a weak $(\kappa,q)$-Poincaré inequality with some $\kappa>q$. By Hölder's inequality, the left hand side of the weak $(\kappa,q)$-Poincaré inequality can be estimated from below by replacing $\kappa$ with any positive $\kappa'<\kappa$. Hence we conclude, that if $X$ supports a weak $(1,p)$-Poincar\'e inequality with $1<p<\infty$, then $X$ also supports a weak $(q,q)$-Poincar\'e inequality with some $1<q<p$.

\end{remark}


\subsection{Parabolic Newtonian spaces}

For $1<p<\infty$, we say that
\[
u\in L^p(0,T;N^{1,p}(\Om)),
\]
if the function $x
\mapsto u(x,t)$ belongs to $N^{1,p}(\Om)$ for almost every $0 < t <
T$, and $u(x,t)$ is measurable as a mapping from $(0,T)$ to $N^{1,p}(\Om)$, that is, the preimage on $(0,T)$ for any given open set in $N^{1,p}(\Om)$ is measurable. 
Furthermore, we require that the norm
\[
\|u\|_{L^p(0,T;N^{1,p}(\Omega))}=\left(\int_{0}^{T}\norm{u}^p_{N^{1,p}(\Om)}dt\right)^{1/p}
\]
is finite. Analogously, we define 
$L^p(0,T;N_0^{1,p}(\Om))$ and $L_{\textrm{loc}}^p(0,T;N_{\textrm{loc}}^{1,p}(\Omega))$. The space of compactly supported Lipschitz-continuous functions $\Lip_c(\Om_T)$ consists of functions $u,\, \textrm{supp}\, u\subset \Om_T$, for which there exists a positive constant $C_{\Lip}(u)$ such that
\[
\begin{split}
\abs{u(x,t)-u(y,s)}\leq C_{\Lip}(u) (d(x,y)+\abs{t-s}),
\end{split}
\]
whenever $(x,t),(y,s) \in \Om_T$.
The parabolic minimal $p$-weak upper gradient of a function $u\in L_{\textrm{loc}}^p(t_1,t_2;N_{\textrm{loc}}^{1,p}(\Omega))$ is defined in a natural way by setting
\begin{align*}
g_u(x,t)=g_{u(\cdot,t)}(x), 
\end{align*}
at $\nu$-almost every $(x,t)\in \Omega\times (0,T)$. When $u$ depends on time, we refer to $g_u$ as the upper gradient of $u$. The next Lemma on taking limits of upper gradients will be used later in this paper. Here and throughout this paper we denote the time wise mollification of a function by
\begin{align*}
f_\varepsilon(x,t)=\int_{-\varepsilon}^\varepsilon \zeta_\varepsilon(s)f(x,t-s)\,ds,
\end{align*}
where $\zeta_\varepsilon$ is  the standard mollifier with support in $(-\varepsilon,\varepsilon)$.
\begin{lemma}\label{convergence of upper gradient} Let $u\in L_\textrm{loc}^p(0,T;N_{\textrm{loc}}^{1,p}(\Omega))$. Then the following statements hold:
\begin{itemize} \item[(a)] As $s\rightarrow 0$, we have $g_{u(x,t-s)-u(x,t)}\rightarrow 0$ in $L_{\textrm{loc}}^p(\Omega_T)$.
\item[(b)] As $\varepsilon\rightarrow 0$, we have $g_{u_\varepsilon -u} \rightarrow 0$ pointwise $\nu$-almost everywhere in $\Omega_T$ and in $L_{\textrm{loc}}^p(\Omega_T)$.
\end{itemize}
\begin{proof}
See Lemma 6.8 in \cite{MassSilj11}. 
\end{proof}
\end{lemma}

The following density result concerning parabolic Newtonian spaces is often cited when working with parabolic quasiminimizers. Here we give a direct proof for $1<p<\infty$.
\begin{lemma}\label{density lemma}
Assume $(X,d,\mu)$ is a complete doubling space supporting a $(1,p)$-Poincar\'e inequality, where $1<p<\infty$.  Let $\phi \in L^p(0,T; N_0^{1,p}(\Omega))$. Then for every $\varepsilon>0$ there exists a function $\varphi \in \textrm{Lip}(\Omega_T)$ such that $\{\varphi \neq 0\}\subset \subset \Omega_T$ and
\begin{align*}
\|\phi-\varphi\|_{L^p(0,T;N^{1,p}(\Omega))}<\varepsilon,\quad \textrm{ and }\quad \nu(\{\varphi\neq 0\}\setminus \{\phi \neq 0\})<\varepsilon.
\end{align*}
\begin{proof}
Assume $\phi$ is as in the claim, and let $\varepsilon>0$. The measure $\nu$ is regular in $\Omega_T$, and so there exists an open set $F\subset \Omega$ such that $\{\phi\neq 0\}\subset F$ and $\nu(F\setminus \{\phi \neq 0\})<\varepsilon$. 
For each $t\in (0,T)$, denote the open set $F_t=\{x\,:\,(x,t)\in F\}$.
Then for almost every $t\in (0,T)$ we have $\phi(\cdot,t)\in N_0^{1,p}(F_t)$. Since by assumption $(X,d,\mu)$ is a complete doubling metric space which supports a $(1,p)$-Poincaré  inquality, it is a result of Shanmugalingam \cite{Shan01} that the space of compactly supported Lipschitz continuous functions is dense in $N_0^{1,p}$. Hence for almost every $t\in(0,T)$ there exists a function $\widetilde \psi(\cdot,t) \in \textrm{Lip}_\textrm{c}(F_t)$ such that
\begin{align}\label{spatial approximation}
\|\phi(\cdot,t)-\widetilde\psi(\cdot,t) \|_{N^{1,p}(\Omega)}<\varepsilon.
\end{align}
Denote for each $\delta>0$, 
\begin{align*}
E_\delta&=\left\{\,t\in(\delta,T-\delta)\;:\,d(\{\widetilde \psi(\cdot,t)\neq 0\},\Omega\setminus F_t) \geq \delta\,\right.\\
&\left.\qquad\qquad \,C_\textrm{Lip}(\widetilde \psi(\cdot,t))\leq\delta^{-1},\qquad \|\widetilde \psi(\cdot,t)\|_\infty\leq \delta^{-1}\,\right\}.
\end{align*}
Since each $\widetilde \psi(\cdot,t)$ has compact support, is bounded and Lipschitz-continuous, we can see that $|(0,T)\setminus E_\delta|\rightarrow 0$ as $\delta \rightarrow 0$. Since $\phi \in L^p(0,T; N_0^{1,p}(\Omega))$, this implies that we can fix $\delta$ to be such that
\begin{align}\label{residue approximation}
\|\phi\|_{ L^p((0,T)\setminus E_\delta ; N^{1,p}(\Omega))}<\varepsilon.
\end{align}
For each $t\in E_\delta$, define the compact set
\begin{align*}
K_t=\left\{ x\in F_t\,:\, d(x,\Omega \setminus F_t)\geq \delta \right\}.
\end{align*}
The fact that $K_t$ is a compact subset of $F_t$ implies that for each $t\in E_\delta$ the set $K_t\times \{t\}$ is a compact subset of $F$. This implies that for each $t\in E_\delta$ there exists a positive constant $0<\delta_t<\delta$ such that 
$\overline {B(x,\delta_t)}\times [t-\delta_t,t+\delta_t] \subset F$ for every $x\in K_t$. 
Denote then for every $0<\rho <\delta$,
\begin{align*}
\widetilde E_{\rho}= \{t\in E_\delta \,:\, \overline {B(x,\rho)}\times [t-\rho,t+\rho] \subset F \textrm{ for every }x\in K_t\}.
\end{align*} 
Then $|E_\delta\setminus \widetilde E_{\rho}|\rightarrow 0$ as $\rho \rightarrow 0$, and so we can fix $\rho<\delta$ to be such that
\begin{align}\label{residue approximation 2}
\|\phi\|_{ L^p(E_\delta\setminus \widetilde E_\rho ; N^{1,p}(\Omega))}<\varepsilon.
\end{align}
Set 
\begin{align*}
\psi(x,t)=
\begin{cases}
\widetilde \psi(x,t), &\,t\in \widetilde E_\rho,\\
0, &\textrm{otherwise}.
\end{cases}
\end{align*}
By \eqref{spatial approximation}, \eqref{residue approximation} and \eqref{residue approximation 2}, we have 
\begin{align*}
&\|\phi-\psi\|_{L^p(0,T;N^{1,p}(\Omega)}\leq \|\phi-\psi\|_{L^p(E_\delta;N^{1,p}(\Omega))}+\| \phi \|_{L^p((0,T)\setminus E_\delta;N^{1,p}(\Omega))}\\
 &\quad \leq \|\phi-\psi\|_{L^p(\widetilde E_\rho;N^{1,p}(\Omega))} +\|\phi\|_{L^p(E_\delta \setminus \widetilde E_\rho;N^{1,p}(\Omega))} +\|\phi\|_{L^p(E_\delta;N^{1,p}(\Omega))}\\
 &\quad\leq \left(\int_{\widetilde E_\rho} \|\phi-\widetilde \psi\|_{N^{1,p}(\Omega)}^p\,dt\right)^{\frac{1}{p}} + 2\varepsilon\leq T^\frac{1}{p}\varepsilon +2\varepsilon.
\end{align*}
Now, from the way we constructed $\psi$, it follows that for every $(x,t)\in \{\psi\neq 0\}$ we have $\overline {B(x,\rho)} \times [t-\rho,t+\rho]\subset F$, and so $\{\psi\neq 0\}$ is compactly contained in $F$. 
Hence there exists a compact set $K$ such that for any $0<\sigma<\rho$, we also have 
$\{\psi_\sigma\neq 0\}\subset K\subset\subset F$. Since for each $t\in (0,T)$ the function $\psi(\cdot,t)$ is Lipschitz-continuous with Lipschitz constant $\delta^{-1}$, we have for every $x,y\in \Omega$ and $t\in (0,T)$
\begin{align*}
|\psi_\sigma(x,t)-\psi_\sigma(y,t)|\leq \int_{-\sigma}^{\sigma} \zeta_\sigma(s) |\psi(x,t-s)-\psi(y,t-s)|\,ds \leq \delta^{-1} d(x,y).
\end{align*}
On the other hand, it is straightforward to show using the theory of mollifiers, that 
since $|\psi|\leq \delta^{-1}$ uniformly in $\Omega_T$, for a fixed $0<\sigma<\rho$ the time derivative of $\psi_\sigma$ is uniformly bounded in $\Omega_T$. Hence we have for every $x\in \Omega$ and $t_1,t_2\in (0,T)$,
\begin{align*}
|\psi_\sigma(x,t_1)-\psi_\sigma(x,t_2)| \leq \left\| \frac{\partial \psi_\sigma}{\partial t} \right\|_\infty|t_1-t_2|.
\end{align*}
This means that for every $0<\sigma<\rho$  we have $\psi_\sigma \in Lip(\Omega_T)$ and $\{\psi_\sigma \neq 0\} \subset K$. Now,
\begin{equation}\label{approximation}
\begin{split}
&\|\phi-\psi_\sigma\|_{L^p(0,T;N^{1,p}(\Omega))}\\&\quad\leq \|\phi-\psi\|_{L^p(0,T;N^{1,p}(\Omega))}+\|\psi-\psi_\sigma\|_{L^p(0,T;N^{1,p}(\Omega))}\\
&\quad\leq T^\frac{1}{p}\varepsilon +2\varepsilon+ \|\psi-\psi_\sigma\|_{L^p(K)}+\|g_{\psi-\psi_\sigma}\|_{L^p(K)}.
\end{split}
\end{equation}
 Treating $\zeta_\sigma(s)\,ds$ as a unit measure, we have by Jensen's inequality, and Fubini's theorem
\begin{align*}
\int_{K}|\psi-\psi_\sigma|^p\,d\nu \leq \int_{-\sigma}^{\sigma} \zeta_\sigma(s) \int_{K}|\psi(x,t)-\psi(x,t-s)|^p\,d\nu\,ds.
\end{align*}
Therefore, by continuity of translation for functions in $L^p(\Omega_T)$, we see that the third term on the right hand side of \eqref{approximation} tends to zero  as $\sigma\rightarrow 0$. Since for every $t\in (0,T)$ we have $\psi(\cdot,t)\in\,$Lip$_c(\Omega)$ uniformly with the same Lipschitz constant $\delta^{-1}$, by the proof of Lemma 6.8 in \cite{MassSilj11} (note that in that proof, for Lipshitz-functions no density results are used), also the last term on the right hand side of \eqref{approximation} tends to zero as $\sigma \rightarrow 0$. Moreover, for every $0<\sigma<\rho$ we have
\begin{align*}
\nu(\{\psi_\sigma \neq 0\}\setminus \{\phi \neq 0\})\leq \nu(F\setminus \{\phi \neq 0\}) \leq \varepsilon.
\end{align*}
Setting $\varphi=\psi_\sigma$ with a small enough $\sigma$ completes the proof. 
\end{proof}
\end{lemma}
For the case $1<p<2$, it is not obvious that convergence of the dense Lipshitz  functions in the parabolic Newtonian space implies convergence also in $L^2(\Omega_T)$. In the following corollary we check that if the limit function is in $L^2(\Omega_T)$, then this is the case.
\begin{corollary}\label{fortified convergence} Assume $(X,d,\mu)$ is a complete doubling space supporting a $(1,p)$-Poincar\'e inequality, where $1<p<\infty$.  Let $\phi \in L^p(0,T; N_0^{1,p}(\Omega))\cap L^2(\Omega_T)$. For every $\varepsilon>0$ there exists a function $\varphi \in \textrm{Lip}(\Omega_T)$ such that $\{\varphi \neq 0\}\subset \subset \Omega_T$ and
\begin{align*}
&\|\phi-\varphi\|_{L^p(0,T;N^{1,p}(\Omega))}<\varepsilon, \qquad \|\phi-\varphi\|_{L^2(\Omega_T)}<\varepsilon,\\
&\textrm{ and }\quad \nu(\{\varphi\neq 0\}\setminus \{\phi \neq 0\})<\varepsilon.
\end{align*}
\begin{proof} If $p\geq 2$, the claim is clearly true by Lemma \ref{density lemma} and Hölder's inequality. So let $1<p<2$, and Let $\phi \in L^p(0,T; N_0^{1,p}(\Omega))\cap L^2(\Omega_T)$. Let $\varepsilon>0$. For each $k$, denote the cutoff function
\begin{align*}
\phi_k(x,t)=
\begin{cases}
k,&\textrm{when }\phi(x,t)\geq k\\
\phi(x,t),& \textrm{when } -k<\phi(x,t)<k\\
-k,&\textrm{when }\phi(x,t)\leq -k.
\end{cases}
\end{align*}
Then  there exists a positive constant $k$ such that 
\begin{align*}
\|\phi-\phi_k\|_{L^p(0,T;N^{1,p}(\Omega))}<\frac{\varepsilon}{2} \textrm{ and } \|\phi-\phi_k\|_{L^2(\Omega_T)}<\frac{\varepsilon}{2}.
\end{align*}
By Lemma \ref{density lemma}, there exists a sequence of compactly supported functions $(\varphi_i)_i\subset \Lip(\Omega_T)$  such that
\begin{align*}
\|\phi_k-\varphi_i\|_{L^p(0,T;N^{1,p}(\Omega_T))}\rightarrow 0\quad \textrm{ and }\quad \nu(\{\varphi_i\neq 0\}\setminus \{\phi_k\neq 0\})\rightarrow 0
\end{align*}
as $i \rightarrow \infty$. Moreover, since $\phi_k$ is bounded, we may assume the sequence $\varphi_i$ to be uniformly bounded by a positive constant $m$. We can now write
\begin{align*}
&\|\phi_k-\varphi_i\|_{L^2(\Omega_T)}^2\\
&= \int_{\Omega_T}|\phi_k-\varphi_i|^2\chi_{\{|\phi_k-\varphi_i|>1\}}\,d\nu+\int_{\Omega_T}|\phi_k-\varphi_i|^2\chi_{\{|\phi_k-\varphi_i|\leq1\}}\,d\nu\\
&\leq (k+m)^2\nu(\{\,\Omega_T\,:\,|\phi_k-\varphi_i|>1\,\})+\int_{\Omega_T}|\phi_k-\varphi_i|^p\,d\nu.
\end{align*}
Since $\varphi_i$ converges to $\phi_k$ in $L^p(\Omega_T)$, the above tends to zero  as $i\rightarrow \infty$. We have
\begin{align*}
\|\phi-\varphi_i\|_{L^p(0,T;N^{1,p}(\Omega))}\leq \frac{\varepsilon}{2}+\|\phi_k-\varphi_i\|_{L^p(0,T;N^{1,p}(\Omega))},
\end{align*}
and on the other hand
\begin{align*}
\|\phi-\varphi_i\|_{L^2(\Omega_T)}\leq \frac{\varepsilon}{2}+\|\phi_k-\varphi_i\|_{L^2(\Omega_T)}.
\end{align*}
Moreover, we have
\begin{align*}
\nu(\{\varphi_i\neq 0\}\setminus \{\phi\neq 0\})=\nu(\{\varphi_i\neq 0\}\setminus \{\phi_k\neq 0\}),
\end{align*}
and so taking $\varphi_i$ with a large enough $i$ completes the proof.
\end{proof}
\end{corollary}
\begin{remark}Notice that if $\phi$ is compactly supported in $\Omega_T$, then Lemma \ref{density lemma} and Corollary \ref{fortified convergence} are also true with supp$\,\phi$ in place of $\{\phi \neq 0\}$, and respectively for the approximating functions. The proofs just have to be repeated after having replaced $\{\phi \neq 0\}$ with supp$\,\phi$, respectively for each function. Hence  we have that for $\phi \in L^p(0,T; N_0^{1,p}(\Omega))\cap L^2(\Omega_T)$ such that $\textrm{supp}\,\phi\subset \subset \Omega_T$ and $\varepsilon>0$, there exists a function $\varphi \in \textrm{Lip}(\Omega_T)$ such that $\textrm{supp}\,\varphi\subset \subset \Omega_T$,
\begin{align*}
&\|\phi-\varphi\|_{L^p(0,T;N^{1,p}(\Omega))}<\varepsilon, \qquad \|\phi-\varphi\|_{L^2(\Omega_T)}<\varepsilon,\\
&\textrm{ and }\quad \nu(\textrm{supp}\,\varphi\setminus \textrm{supp}\,\phi)<\varepsilon.
\end{align*}.
\end{remark}


\subsection{Parabolic quasiminimizers}
\begin{definition}
Let $\Omega$ be an  open subset of $X$, $u:
\Omega\times (0,T)\to \R$ and $K'\geq 1$. Let $1<p<\infty$. A function $u$
belonging to the parabolic space
$L^p_{\trm{loc}}(0,T;N_{\trm{loc}}^{1,p}(\Omega))$ is a
\emph{parabolic quasiminimizer} if
\begin{equation*}
\begin{split}
\int_{\{\phi\neq 0\}}u\parts{\phi}{t} \,d \nu + E(u, \{\phi \neq 0\}) \,\leq K'  E(u+\phi, \{\phi \neq 0\}),
\end{split}
\end{equation*}
for every $\phi \in \textrm{Lip}(\Omega_T)$  such that $\{\phi\neq 0\}\subset \subset \Omega_T$, where
we denote
\begin{align*}
E(u,A)=\int_{A}F(x,t,g_u)\,d \nu,
\end{align*}
 and $F: \Omega \times (0,T) \times \R \to \R $ satisfies the following assumptions:
\begin{enumerate}
\item $ (x,t)\mapsto F(x,t,\xi)$ is
  measurable for every $\xi$,
\item $\xi \mapsto F(x,t,\xi)$ is continuous for almost every $(x,t)$,
\item  there exist $0<c_1\leq c_2<\infty$ such that for every
  $\xi$ and almost every $(x,t)$, we have
\begin{equation*}
c_1 \abs{\xi}^p \leq F(x,t,\xi)\leq c_2
\abs{\xi}^p.\end{equation*}
\end{enumerate}
\end{definition}
As a consequence of the above, a parabolic quasiminimizer $u$ satisfies
\begin{align}\label{qmin}
\alpha \int_{\{\phi\neq 0\}} u \frac{\partial \phi}{\partial t}\,d\nu+\int_{\{\phi\neq 0\}} g_u^p\,d\nu \leq K \int_{\{\phi\neq 0\}} g_{u+\phi}^p\,d\nu,
\end{align}
with $K=c_2c_1^{-1}K'\geq 1$ and $\alpha=c_1^{-1}$, for every $\phi \in \textrm{Lip}(\Omega_T)$  such that $\{\phi\neq 0\}\subset \subset \Omega$. There is a subtle difficulty in proving an energy estimate
for parabolic quasiminimizers: one often needs a test function
depending on $u$ itself, but $u$ is a priori not necessarily in $\Lip(\Om_T)$ nor has compact support. We treat this difficulty
in the following manner. Consider a test function $\phi \in \Lip(\Omega_T)$ with compact support. Plugging in the test function $\phi(x,t+s)$ and conducting a change of variable in \eqref{qmin}, we see that there exists a constant $\varepsilon>0$ such that for every $-\varepsilon<s<\varepsilon$,
\begin{align*}
\alpha\int_{\{\phi\neq 0\}} u(x,t-s) \frac{\partial \phi}{\partial t}\,d\nu+\int_{\{\phi\neq 0\}} g_{u(x,t-s)}^p\,d\nu 
\leq K \int_{\{\phi\neq 0\}} g_{u(x,t-s)+\phi}^p\,d\nu.
\end{align*}
Let now $\zeta_{\eps}(s)$ be a standard
mollifier whose support is contained in $(-\eps,\eps)$. We multiply the above inequality with $\zeta_{\eps}(s)$ and integrate on both sides with respect to $s$, use Fubini's theorem to change the order of integration, and lastly use partial integration for the first term on the left hand side, to obtain
\begin{equation}\label{quasiminimizer inequality}
\begin{split}
-\alpha\int_{\{\phi\neq 0\}} \frac{\partial u_\varepsilon}{\partial t} \phi \,d\nu+\int_{\{\phi\neq 0\}}\left(g_{u}^p\right)_\eps \,d\nu 
\leq K \int_{\{\phi\neq 0\}} \left( g_{u(x,t-s)+\phi}^p\right)_\eps\,d\nu,
\end{split}
\end{equation}
for every compactly supported $\phi\in \Lip(\Omega_T)$. In the next section, we will derive the energy estimates starting from this last expression. We end this section by showing the following equivalence concerning the class of admissible test functions.

\begin{lemma} \label{density corollary} Assume $(X,d,\mu)$ is a complete doubling space supporting a $(1,p)$-Poincar\'e inequality, where $1<p<\infty$. Assume $u\in L^p_\textrm{loc}(0,T;N_{\textrm{loc}}^{1,p}(\Omega))$ is a parabolic quasiminimizer. Then inequality \eqref{quasiminimizer inequality} holds for every $\phi \in \Lip(\Omega_T)$, such that $\{\phi\neq 0\}\subset \subset \Omega_T$, if and only if it holds for every $\phi \in L^p(0,T;N^{1,p}(\Omega))\cap L^2(\Omega_T)$ such that $\{\phi\neq 0\}\subset \subset \Omega_T$.
\begin{proof}
Only the other direction needs a proof. Assume \eqref{quasiminimizer inequality} holds for every compactly supported $\phi \in \Lip(\Omega_T)$. Let $\psi \in L^p(0,T;N^{1,p}(\Omega))\cap L^2(\Omega_T)$ be  such that $\{\psi\neq 0\}\subset \subset \Omega_T$. Let $\varepsilon>0$. Let $\phi$ be a function in $\Lip(\Omega_T)$. 
By adding and substracting, and using Hölder's inquality, we can write the left hand side of \eqref{quasiminimizer inequality} as
\begin{align*}
&-\alpha\int_{\{\psi\neq 0\}} \frac{\partial u_\varepsilon}{\partial t} \psi \,d\nu+\int_{\{\psi\neq 0\}}\left(g_{u}^p\right)_\eps \,d\nu 
\leq \alpha \left(\int_{\{\phi\neq 0\}\cup \{\psi \neq 0\}}  \left|\frac{\partial u_\varepsilon}{\partial t}\right|^2\,d\nu\right)^\frac{1}{2}\\
&\cdot \left(\int_{\{\phi\neq 0\}\cup \{\psi \neq 0\}} |\psi-\phi|^2\,d\nu\right)^\frac{1}{2}-\left(\int_{\{\phi\neq 0\}\setminus\{\psi\neq 0\} }g_{u}^p \,d\nu \right)_\eps\\
&-\alpha\int_{\{\phi\neq 0\}} \frac{\partial u_\varepsilon}{\partial t} \phi \,d\nu+\int_{\{\phi\neq 0\}}\left(g_{u}^p\right)_\eps \,d\nu.
\end{align*}
On the other hand, using Minkowski's inequality we can write
\begin{align*}
 &\int_{\{\phi\neq 0\}} \left( g_{u(x,t-s)+\phi}^p\right)_\eps\,d\nu\leq  \left(\left(\left(\int_{\{\phi\neq 0\}}  (g_{u(x,t-s)+\psi}+g_{\psi-\phi})^p\,d\nu\right)^\frac{1}{p}\right)^p \right)_\varepsilon \\
&\quad\leq \left(\left(\left( \int_{\{\phi\neq 0\}\setminus \{\psi\neq 0\}}  g_{u(x,t-s)+\psi}^p\right)^\frac{1}{p}+ \left(\int_{\{\psi\neq 0\}}  g_{u(x,t-s)+\psi}^p\right)^\frac{1}{p}\right.\right.\\
&\qquad\left.\left.+ \left(\int_{\{\phi\neq 0\}} g_{\psi-\phi}^p\,d\nu \right)^\frac{1}{p}\right)^p\right)_\varepsilon.
\end{align*}
Since the above two expressions and inequality \eqref{quasiminimizer inequality} hold for every compactly supported $\phi \in \Lip(\Omega_T)$, by Corollary \ref{fortified convergence} we can take a sequence $(\phi_i)_i \subset \Lip(\Omega_T)$ such that
\begin{align*}
&\|\psi-\phi_i\|_{L^p(0,T;N^{1,p}(\Omega))}\rightarrow 0,\quad \|\psi-\phi_i\|_{L^2(\Omega_T)}\rightarrow 0\\
&\textrm{and }\quad \nu(\{\varphi_i\neq 0\}\setminus \{\psi \neq 0\})\rightarrow 0,
\end{align*}
as $i \rightarrow \infty$. Hence we see that inequality \eqref{quasiminimizer inequality} holds for $\psi$, with the same constant $K$.
\end{proof}
\end{lemma}



\section{Estimates}
\label{sec:basic-estimates}

From now on we assume that $p=2$, and that $u\in L_{\textrm{loc}}^2(0,T;N_\textrm{loc}^{1,2}(\Omega))$ is a parabolic $K$-quasiminimizer. In this section we deduce from the definition of parabolic quasiminimizers a fundamental energy estimate for $u$. From this energy estimate we obtain a Caccioppoli inequality by the so called hole filling iteration. By combining the Caccioppoli inequality, a parabolic version of Poincar\'e's inequality derived from the fundamental energy estimate, and Sobolev's inequality,  we obtain a reverse Hölder inequality for $g_u$. Since $p=2$, no difficulties arise when combining the estimates together in proving the reverse Hölder inequality. Higher integrability then follows from the reverse Hölder inequality by a modification of Gehring's famous lemma.

We begin by establishing a fundamental energy estimate by choosing a suitable test function in the definition of parabolic quasiminimizers.
\begin{lemma}[Fundamental energy estimate]
\label{energy estimate near the initial boundary}
There exists a positive constant $
c=c(K)$, such that for every $Q_{\rho}=B_\rho(x_0) \times \Lambda_\rho(t_0)$, $\rho<\sigma$ such that $Q_\sigma\subset \Omega_T$, we have
\[
\begin{split}
 &\esssup_{t \in
  \Lambda_{\rho}} \int_{B_{\rho}} |u-u_{\sigma}(t)|^2 \,d \mu +\int_{Q_{\rho}}  g_u^2 \,d \nu \\
&\leq
c \int_{(\q{\sigma}\setminus \q{\rho})} g_u^2 \,d \nu
+\frac{c}{(\sigma-\rho)^2}\frac{\mu(B_\sigma)}{\mu(B_\rho)}\int_{\q{\sigma}}|u-u_{\sigma}(t)|^2
\,d \nu.
\end{split}
\]

\begin{proof}
Assume $Q_\rho=B_\rho(x_0)\times \Lambda_\rho(t_0)$, and $\rho<\sigma$ are such that $B_\sigma(x_0) \subset \Omega$ and $\Lambda_\rho \subset (0,T)$. Let $t'\in \Lambda_\rho(t_0)$. Define
 \begin{align*}
 &\chi_h(t) = 
 \begin{cases}
 0,& t>t'\\
\frac{t'-t}{h},&t'-h\leq t \leq t',\\
 1,& t \leq t'-h.
\end{cases}
 \end{align*}
Let $\varphi_1\in N^{1,2}(B_\sigma)$, $0\leq \varphi_1\leq 1$, be such that  $\varphi_1=1$ in $B_\rho(x_0)$, $\spt \varphi_1$ is a compact subset of  $B_\sigma(x_0)$, and
\begin{align*}
g_{\varphi_1}^2\leq \frac{c}{(\sigma-\rho)^2}.
\end{align*}
Define then $\varphi_2\in \textrm{Lip}(0,T)$ such that $0\leq \varphi_2 \leq 1$ and $|\partial \varphi_2 /\partial t| \leq c(\sigma-\rho)^{-2}$, by setting
\begin{align*}
 \varphi_2(t) = 
 \begin{cases}
 1,&  t\geq t_0-\frac{\rho^2}{2},\\
 \frac{2t-2t_0+\sigma^2}{\sigma^2-\rho^2}, & t_0-\frac{\sigma^2}{2}< t< t_0-\frac{\rho^2}{2},\\
 0, &   t\leq t_0-\frac{\rho^2}{2}.
\end{cases} 
\end{align*}
Set now $\varphi(x,t)=\varphi_1(x)\varphi_2(t)$.
For a function $f(x,t)$, define for every $t\in (0,T)$
\begin{align}\label{weighted average}
f_{\sigma}^\varphi(t)=
\frac{\int_{B_\sigma} f(x,t) \varphi(x,t)\, d\mu}{\int_{B_\sigma} \varphi(x,t) \, d\mu}=\frac{\int_{B_\sigma} f(x,t) \varphi_1(x)\, d\mu}{\int_{B_\sigma} \varphi_1(x) \, d\mu},
\end{align}
where when $t\not \in\Lambda_\sigma(t_0)$ we use the right hand side expression as  the definition. Choose $\phi=\varphi(u_{\eps}-(u_\varepsilon)_\sigma^\varphi)\chi_{h}$. Since $u$ is a parabolic quasiminimizer  and $\phi\in L^2(0,T;N^{1,2}(\Omega))$ is compactly supported, by Lemma \ref{density corollary}  we can insert $\phi$ as a test function into inequality \eqref{quasiminimizer inequality} and examine the resulting terms. In the first term on the left hand side, after performing partial integration, we add and subtract
$(u_\eps)_{\sigma}^\varphi (\partial \phi/\partial t)$ to obtain
\begin{equation}\label{valivaihe}
\begin{split}
&\int_{\Om_T} u_{\eps}  \parts{\phi}{t} \,d \nu=\int_{\Om_T}
(u_{\eps}-(u_\varepsilon)_\sigma^\varphi(t)) \parts{\phi }{t} \,d \nu+\int_{\Om_T}(u_\varepsilon)_\sigma^\varphi(t)\parts{\phi }{t} \,d \nu.
\end{split}
\end{equation}
Integrating by parts and using the definition of
$(u_\varepsilon)_\sigma^\varphi(t)$, we
  see that the last term on the right hand side vanishes
\[
 \begin{split}
&\int_{\Om_T}(u_\varepsilon)_\sigma^\varphi(t) \parts{\phi}{t} \,d \nu\\
&=
 -\int_{t_0-\frac{\sigma^2}{2}}^{t'} \parts{}{t} (u^{\varphi}_{\sigma,\eps}(t))\l(\int_{B_{\sigma}}u_{\eps} \varphi \,d \mu - \frac{\int_{B_{\sigma}}
 \varphi\,d \mu \int_{B_{\sigma}} u_{\eps} \varphi \,d
 \mu}{\int_{B_{\sigma}}\varphi \,d \mu}\r) \chi_{h}(t)\,d t
=0.
 \end{split}
\]
Then we integrate the  first term on the right side of \eqref{valivaihe}
by parts, and so we obtain for every $h$ small enough
\begin{equation*}
\begin{split}
&\int_{\Om_T} u_{\eps}  \parts{\phi}{t} \,d \nu=\frac{1}{2}\int_{\Om_T}
(u_{\eps}-(u_\varepsilon)_\sigma^\varphi(t))^2\parts{}{t}\left(\varphi\chi_{h}\right) \,d \nu
\\
&=\frac{1}{2h} \int_{t'-h}^{t'}
\int_{B_{\sigma}}\abs{u_\eps(x,t)-(u_\eps)_{\sigma}^{\varphi}(t)}^2\varphi
 \,d \mu\,dt\\
 &\quad-\frac{1}{2}\int_{\Om_T}
(u_{\eps}-(u_\varepsilon)_\sigma^\varphi(t))^2\chi_h\parts{\varphi}{ t} \,d \nu.
\end{split}
\end{equation*}
Now,  letting first $\eps\to 0$ and then $h\to 0$, we obtain
\begin{align*}
&\liminf_{\varepsilon,h\rightarrow 0} \int_{\Om_T} u_{\eps} \parts{ \phi}{t} \,d \nu\\
&\quad\geq\frac{1}{2} 
\int_{B_{\sigma}}\abs{u(x,t')-u_{\sigma}^{\varphi}(t')}^2\varphi(x,t')\,d \mu
-\frac{1}{2}\int_{\Om_T}
(u-u_\sigma^\varphi(t))^2\left|\parts{\varphi}{ t}\right| \,d \nu\\
&\quad\geq \frac{1}{2} 
\int_{B_{\sigma}}\abs{u(x,t')-u_{\sigma}^{\varphi}(t')}^2\varphi(x,t')\,d \mu-\frac{1}{(\sigma-\rho)^2}\int_{Q_\sigma}
(u-u_\sigma^\varphi(t))^2 \,d \nu.
\end{align*}
On the right hand side of inequality \eqref{quasiminimizer inequality}, we note that for every $h,\eps>0$ small enough, in the set $\{\phi\neq 0\}$ we have
\begin{align*}
&(g_{u(\cdot,\cdot-s)-\varphi(u_\eps-(u_\eps)_\sigma^\varphi)\chi_{h}}^2)_\eps\leq c (g^2_{u(\cdot,\cdot-s)-u})_\eps+ cg^2_{u-\varphi(u-u_\sigma^\varphi)}\\&\qquad+cg^2_{\varphi(u-u_\sigma^\varphi)}(1-\chi_{h})^2
+cg_\varphi^2((u_\eps)_\sigma^\varphi-u_\sigma^\varphi)^2\chi_{h}^2\\&\qquad+c(u-u_\eps)^2 g_{\varphi}^2 \chi_{h}^2+c\varphi^2 g_{u-u_\eps}^2 \chi_{h}^2.
\end{align*}
By Lemma \ref{convergence of upper gradient}, we know that $g_{u-u_\eps}^2\rightarrow 0$  and $g_{u(\cdot,\cdot-s)-u}^2\rightarrow 0$ in $L_{\textrm{loc}}^1(\Omega_T)$ as $\eps\rightarrow 0$. Also, by \eqref{weighted average} we have for every small enough $\varepsilon$
\begin{align*}
\int_{\{\phi\neq 0\}}(u_\sigma^\varphi-(u_\eps)_\sigma^\varphi)^2\, d\nu\leq \mu(\Omega)\int_{t_0-\frac{\sigma^2}{2}}^{t'}(u_\sigma^\varphi(t)-(u_\sigma^\varphi)_\eps(t))^2\,dt\\
\leq \mu(\Omega)\int_{-\varepsilon}^{\varepsilon} \int_{t_0-\frac{\sigma^2}{2}}^{t'}|u_\sigma^\varphi(t)-u_\sigma^\varphi(t-s)|^2\,dt \, \zeta_{\eps}(s) \,ds,
\end{align*}
and therefore the fact that $u_{\sigma}^\varphi\in L^2(0,T)$ implies by the continuity of translation, that the above expression tends to zero as $\eps \rightarrow 0$. We obtain
\begin{align*}
\limsup_{\varepsilon,  h\rightarrow 0}\int_{\{\phi \neq 0\}} \left(g_{u(\cdot,\cdot-s)-\phi}^2 \right)_\varepsilon\,d\nu \leq c\int_{Q_\sigma}g_{u-\varphi(u-u_\sigma^\varphi)}^2 \, d\mu\, dt.
\end{align*}
Next we note that since $u_\sigma^\varphi$ does not depend on $x$, and hence its upper gradient vanishes, we can write, after noting that $\phi=1$ in $Q_\rho$,
\begin{align*}
\int_{Q_\sigma}g_{u-\varphi(u-u_\sigma^\varphi)}^2 \, d\nu \leq c\int_{Q_\sigma}|1-\varphi|^2g_{u}^2 \, d\nu +c\int_{Q_\sigma}|u-u_\sigma^\varphi(t)|^2g_{\varphi}^2 \, d\nu\\
\leq c\int_{Q_\sigma\setminus Q_\rho}g_{u}^2 \, d\nu+\frac{c}{(\sigma-\rho)^2}\int_{\q{\sigma}}|u-u_{\sigma}^\varphi(t)|^2
\,d \nu.
\end{align*}
Since $t'$ was assumed to be arbitrary in $\Lambda_\rho(t_0)$ and the constants in the estimates are independent of $t'$, combining the above expressions through \eqref{quasiminimizer inequality}, and remembering the definition of $\varphi$, leads us to the estimate
\[
\begin{split}
 &\esssup_{t \in
  \Lambda_{\rho}} \int_{B_{\rho}} |u-u_{\sigma}^{\varphi}(t)|^2 \,d \mu +\int_{Q_{\rho}}  g_u^2 \,d \nu \\
&\quad\leq
c \int_{(\q{\sigma}\setminus \q{\rho})} g_u^2 \,d \nu
+\frac{c}{(\sigma-\rho)^2}\int_{\q{\sigma}}|u-u_{\sigma}^\varphi(t)|^2
\,d \nu,
\end{split}
\]
where $c=c(K)$. We complete the  proof by noting that for any $t\in (0,T)$, we have
\begin{align*}
\int_{B_\rho}&|u-u_{\sigma}(t)|^2\,d\mu\\
 &\leq 2\int_{B_\rho}|u-u_{\sigma}^\varphi(t)|^2\,d\mu+2\int_{B_\rho}\left(\vint_{B_{\sigma}}|u_{\sigma}^\varphi(t)-u|^2\,d\mu \right)\,d\mu\\
 &\leq  4\int_{B_\rho}|u-u_{\sigma}^\varphi(t)|^2\,d\mu.
 \end{align*}
On the other hand, by the triangle inequality and by Jensen's inequality, and since $\varphi_1=1$ in $B_\rho$,
 \begin{align*}
 \int_{B_{\sigma}}&|u-u_{\sigma}^\varphi(t)|^2 \, d\mu
\leq 2\int_{B_{\sigma}}|u-u_{\sigma}(t)|^2 \, d\mu\\&+2\int_{B_{\sigma}}\left( \left(\int_{B_{\sigma}}\varphi_1\,d\mu\right)^{-1}\int_{B_{\sigma}}|u_{\sigma}(t)-u|^2\varphi_1 \,d\mu\right) \, d\mu\\
 &\leq 4\frac{\mu(B_{\sigma})}{\mu(B_\rho)} \int_{B_{\sigma}}|u-u_{\sigma}(t)|^2 \, d\mu.
 \end{align*}
\end{proof}
\end{lemma}
Having obtained the energy estimate, we use the so called hole filling iteration \cite{Widm71} to extract a Caccioppoli inequality from it.
\begin{lemma}[Caccioppoli]
\label{caccioppoli near the initial boundary}
There exists a positive constant $c=c(c_\mu,K)$ so that for any $Q_\rho=B_\rho(x_0)\times \Lambda_\rho(t_0)$ such that $Q_{2\rho}\subset \Omega_T$, we have
\[
\label{eq:cacckvasi}
\begin{split}
\int_{Q_{\rho}} g_u^2 \, d \nu &\leq
\frac{c}{\rho^2}\int_{\q{2\rho}}|u-u_{2\rho}(t)|^2  \, d \nu.
\end{split}
\]

\begin{proof} By Lemma \ref{energy estimate near the initial boundary}, for any cylinder $Q_\rho=B_\rho(x_0)\times \Lambda_\rho(t_0)$ such that $Q_{2\rho}\subset \Omega_T$, we have for any $\rho<\sigma\leq 2\rho$,
\[
\begin{split}
 &\esssup_{t \in
  \Lambda_{\rho}} \int_{B_{\rho}} |u-u_{\sigma}(t)|^2 \,d \mu +\int_{Q_{\rho}}  g_u^2 \,d \nu \\
&\leq
c \int_{(\q{\sigma}\setminus \q{\rho})} g_u^2 \,d \nu
+\frac{c}{(\sigma-\rho)^2}\frac{\mu(B_{\sigma})}{\mu(B_\rho)}\int_{\q{\sigma}}|u-u_{\sigma}(t)|^2
\,d \nu
\end{split}
\]
where $c=c(K)$. We add $c\int_{Q_\rho} g_u^2\,d\nu$ to both sides of the expression, and divide by $1+c$, to obtain
 \begin{align*}
 &\int_{Q_\rho} g_u^2 \, d\nu\leq  \frac{c}{1+c} \int_{Q_{\sigma}} g_u^2 \,d\nu+\frac{c}{(1+c)(\sigma-\rho)^2}\frac{\mu(B_{\sigma})}{\mu(B_\rho)}\int_{Q_{\sigma}}|u-u_{\sigma}(t)|^2
\, d \nu.
 \end{align*}
Then we choose 
 \begin{align*}
 \rho_0=\rho,\quad \rho_i-\rho_{i-1}=\frac{1-\beta}{\beta}\beta^i \rho, \quad i=1,2,\dots,k,\quad \beta^2=\frac{1}{2}\left(\frac{c}{1+c}+1\right),
 \end{align*}
 replace $\rho$ by $\rho_{i-1}$ and $\sigma$ by $\rho_i$, and iterate, to have
 \begin{align*}
 \int_{Q_\rho} g_u^2 \, d\nu&\leq  \left(\frac{c}{1+c}\right)^k \int_{Q_{\rho_k}} g_u^2 \,d\nu\\
 &+\sum_{i=1}^{k}\left(\frac{c}{1+c}\right)^i\frac{\mu(B_{\rho_i})}{\mu(B_{\rho_{i-1}})}\frac{1}{(\rho_i-\rho_{i-1})^2}\int_{Q_{\rho_i}}|u-u_{\rho_i}(t)|^2
\, d \nu.
 \end{align*}
 Here among other things $\rho_i\leq 2\rho_{i-1}$ for every $i$, and so by the doubling property  of $\mu$, the ratio $\mu(B_{\rho_i})/\mu(B_{\rho_{i-1}})$ is uniformly bounded. Also, for each $i$ we can estimate  after using Fubini's theorem,
 \begin{align*}
 \int_{Q_{\rho_i}} |u-u_{\rho_i}(t)|^2\,d\nu&\leq 2 \int_{Q_{2\rho}} |u-u_{2\rho}(t)|^2\,d\nu+ 2 \int_{Q_{2\rho}} \vint_{B_{\rho_i}}|u_{2\rho}(t)-u|^2\,d\nu\\
 &\leq 2c \int_{Q_{2\rho}}|u-u_{2\rho}(t)|^2\,d\nu,
 \end{align*}
where $c=c(c_\mu)$. Hence, taking the limit $k\rightarrow \infty$ yields the estimate
 \begin{align*}
\int_{Q_\rho} g_u^2 \, d\nu\leq \frac{c}{\rho^2}\int_{Q_{2\rho}} |u-u_{2\rho}(t)|^2
\, d \nu, 
 \end{align*} 
 where $c=c(c_\mu,K)$.
\end{proof}
\end{lemma}
Now we combine the Caccioppoli inequality, the fundamental energy estimate, Poincar\'e's inequality and Sobolev's inequality together to prove a reverse Hölder inequality type estimate for the upper gradient of $u$. 
\begin{lemma}[Reverse Hölder inequality]\label{reverse hölder away from boundary}
There exists a positive constant $c=c(c_\mu,c_P,\lambda, K)$, and a $1<q<2$, so that for any $Q_\rho=B_\rho(x_0)\times \Lambda_\rho(t_0)$, such that $Q_{2\lambda\rho}\subset \Omega_T$ , we have
 \begin{align*}
 &\vint_{Q_{\rho}} g_u^2 \, d \nu\leq \varepsilon c\vint_{\q{2\lambda\rho}} g_u^2 \,d \nu+2\varepsilon^{-1}c\left(\vint_{Q_{2\lambda \rho}} g_u^q\,d\nu\right)^\frac{2}{q}.
 \end{align*}
 \begin{proof}
 By the Caccioppoli Lemma \ref{caccioppoli near the initial boundary}, by the doubling property of $\mu$ and since $\nu(Q_{\rho})=\rho^2\mu(B_\rho)$, we obtain
 \begin{align*}
\vint_{Q_{\rho}} g_u^2 \, d \nu 
\leq \frac{c}{\rho^4}\int_{\Lambda_\rho}\vint_{B_{2\rho}}|u-u_{2\rho}|^2  \, d \mu\,dt,
 \end{align*}
 where $c=c(c_\mu,c_P,K)$. On the other hand we can write
 \begin{align*}
 &\frac{c}{\rho^4}\int_{\Lambda_\rho}\vint_{B_{2\rho}}|u-u_{2\rho}|^2  \, d \mu\,dt \leq \left(\frac{c}{\rho^2}\esssup_{t \in
  \Lambda_{\rho}} \vint_{B_{\rho}} |u-u_{2\rho}|^2 \, d \mu\right)^{1-\frac{q}{2}}\\
& \qquad\cdot \frac{c}{\rho^2} \int_{\Lambda_\rho}\left(\frac{c}{\rho^2} \vint_{B_{2\rho}}|u-u_{2\rho}|^2  \, d \mu\right)^\frac{q}{2}\,dt\\
&   \leq \left\{\vint_{Q_{2\rho}} g_u^2\,d\nu+c\vint_{\q{2\lambda\rho}} g_u^2 \,d \nu\right\}^{1-\frac{q}{2}} \vint_{Q_{2\lambda \rho}} g_u^q\,d\nu, 
 \end{align*}
 were we used the fundamental energy estimate Lemma \ref{energy estimate near the initial boundary} and the $(2,2)$-Poincaré inequality (see Remark \ref{poincare_remark}) for the former term and the Sobolev type $(2,q)$-Poincaré inequality for the latter term. By the $\varepsilon$-Young inequality we now obtain for every positive $\varepsilon$
 \begin{align*}
 &\vint_{Q_{\rho}} g_u^2 \, d \nu\leq \varepsilon c\vint_{\q{2\lambda\rho}} g_u^2 \,d \nu+\varepsilon^{-1}c\left(\vint_{Q_{2\lambda \rho}} g_u^q\,d\nu\right)^\frac{2}{q},
 \end{align*}
 where $c=c(c_\mu,c_P, \lambda,K)$.
 \end{proof}
\end{lemma}
\section{Establishing higher integrability}
Having established a reverse Hölder inequality for $u$, the remaining part of proving local higher integrability for the upper gradient of $u$ is abstract in nature, and does not make additional use of the assumption that $u$ is a quasiminimizer. We use the following modification of Gehring's lemma.
\begin{theorem}
\label{global gehring}
Let $g\in L^2_{\trm{loc}}(\Omega_T)$  be a non negative measurable functions defined in $\Omega_T$. Let $s$ be the constant from \eqref{eq:DoublingConsequence} and let $q$ be such that $2s/(2+s)<q<2$. Consider a parabolic cylinder $Q_{2R}(z_0)\subset \Omega_T$.  Suppose that there exists a positive constant $A>1$, for which with any $z'=(x',t')$ and $\rho$ such that $Q_{A\rho}(z')\subset Q_{2R}(z_0)$, we have
\begin{equation} 
\label{gekaehto}
\begin{split}
&\vint_{Q_\rho(z')}  g^2 \, d \nu
\leq
\eps\vint_{Q_{A \rho}(z')} g^2
\,d \nu+\gamma\l(\vint_{Q_{A \rho}(z')} g^{{q}} \,d \nu \r)^{2/{q}},
\end{split}
\end{equation}
for any $\eps>0$, where $\gamma$ may depend on $\eps$.
 Then there exists positive constants $\eps_0=\eps_0(c_\mu, A, \gamma, q)$  and $c=c(c_\mu, A,\gamma)$ such that 
\[
\begin{split}
&\left(\vint_{Q_{R}(z_0)} g^{2+\eps} \,d \nu\right)^\frac{1}{2+\eps}\leq c \l(\vint_{Q_{2R}(z_0)} g^2 \,d \nu\r)^{\frac{1}{2}},
\end{split}
\]
for every $0<\eps\leq \eps_0$.
\end{theorem}
%

\begin{proof}
Assume a parabolic cylinder $Q_{2R}$ with center point $z_0=(x_0,t_0)$ such that $Q_{2R}(z_0)\subset \Omega_T$. Define for every $z_1=(x_1,x_2)$, $z_2=(x_2,t_2)\in X\times \R$ the parabolic distance
\begin{align*}
 \textrm{dist}_p(z_1,z_2)=d(x_1,x_2)+|t_1-t_2|^{1/2}.
\end{align*}
Using this, set for every $z\in Q_{2R}$ the functions
\begin{align*}
r(z)&=\,\textrm{dist}_p(z,(X\times \R)\setminus Q_{2R}),\\
 \alpha(z)&=\frac{\nu(Q_{2R})}{\nu(Q_{\frac{r(z)}{5A}}(z))}.
\end{align*}
From the definition of $r(z)$ it can readily be checked that $Q_{r(z)}(z)\subset Q_{2R}$ for every $z\in Q_{2R}$.
For $z\in Q_{2R}$, define
\begin{align*}
h(z)=\alpha^{-1/2}(z)g(z),
\end{align*}  
and for every $\beta>0$, set
\begin{align*}
G(\beta)=\{\,z\in Q_{2R}\cap \Omega_T\,:\,h(z)>\beta\,\}.
\end{align*}
Denote
\begin{align*}
\beta_0=\left( \vint_{Q_{2R}} g^2\,d\nu \right)^{1/2}.
\end{align*}
Assume $\beta>\beta_0$. For $\nu$-almost every $z'\in G(\beta)$, we have for every \\$r\in [r(z')/(5A),r(z')]$, 
\begin{align}\label{larger than}
\vint_{Q_r(z')\cap \Omega_T}g^2\,d\nu \leq \alpha(z') \vint_{Q_{2R}}g^2\,d\nu \leq \alpha(z') \beta^2, 
\end{align}
and by the definition of $G(\beta)$, since $\mu$ is a positive Borel measure,
\begin{align}\label{less than}
\lim_{r\rightarrow 0} \vint_{Q_r(z')} g^2\, d\nu=g^2(z')>\alpha(z') \beta^2.
\end{align}
Now \eqref{larger than} and \eqref{less than} imply that for $\nu$-almost every $z'\in G(\beta)$, there exists a corresponding radius $\rho(z')\in (0,r(z')/(5A))$, for which it holds 
\begin{equation}\label{comparability of balls for g}
\begin{split}
\vint_{Q_{A \rho(z')}(z')} g^2\, d\nu\leq \alpha(z') \beta^2\leq \vint_{Q_{\rho(z')}(z')} g^2\, d\nu. 
\end{split}
\end{equation}
Thus by choosing $\eps=1/2$ in \eqref{gekaehto}, we can absorb the first term on the right hand side of \eqref{gekaehto} into the left hand side and obtain
\begin{align*}
&\vint_{Q_{\rho(z')}(z')} g^2 \, d\nu \leq  \left( c\vint_{Q_{A\rho(z')}(z')}g^q\, d\nu \right)^{2/q},
\end{align*}
for $\nu$-almost every $z'\in G(\beta)$,  where $c=c(\gamma)$. This together with \eqref{comparability of balls for g} yields
\begin{equation}\label{higherint of g when centered in G}
\begin{split}
&\vint_{Q_{5A\rho(z')}(z')} g^2 \, d\nu \leq \left( c\vint_{Q_{A\rho(z')}(z')}g^q\,d\nu \right)^{2/q},
\end{split}
\end{equation}
where $c=c(A,c_\mu,\gamma)$. From the definitions of a parabolic cylinder and the parabolic distance, it follows that
\begin{align*}
2^{-1/2} r(z')\leq r(z) \leq 2 r(z') \qquad \textrm{for every } z\in Q_{r(z')}(z'),\; z'\in Q_{2R}.
\end{align*}
From this it is straightforward to check that
\begin{align*}
\begin{array}{c}Q_{r(z)}(z)\subset Q_{3r(z')}(z'),\\ Q_{r(z')}(z')\subset Q_{4r(z)}(z)\end{array} \quad \textrm{ for every }  z\in Q_{r(z')}(z'),\; z'\in Q_{2R},
\end{align*} 
and so by the doubling property of the measure there exists positive constants $c=c(c_\mu)$, $c'=c'(c_\mu)$ such that
\begin{align}\label{local comparability of alpha}
c\alpha(z')\leq \alpha(z) \leq c' \alpha (z) \qquad \textrm{ for every }z\in Q_{r(z')}(z'),\; z'\in Q_{2R}.
\end{align}
Because of this, we see from \eqref{higherint of g when centered in G} that there exists a positive constant $c=c(A,c_\mu,\gamma)$, such that for $\nu$-almost every $z'\in G(\beta)$, after also using the fact that $\alpha(z)\geq 1$,
\begin{align}\label{higherint of h when centered in G}
\begin{split}
&\vint_{Q_{5A \rho(z')}(z')} h^2 \, d\nu\leq \left( \vint_{Q_{A\rho(z')}(z')}h^q\,d\nu \right)^{2/q}.
\end{split}
\end{align}
On the other hand, by  Hölder's inequality since $1<q<2$, and then by \eqref{local comparability of alpha}, we obtain from \eqref{comparability of balls for g}, 
\begin{align}\label{q boundedness of h}
\begin{split}
&\left( \vint_{Q_{A\rho(z')}(z')}h^q\, d\nu \right)^{(2-q)/q}\leq \left( \vint_{Q_{A\rho(z')}(z')}h^2\, d\nu \right)^{(2-q)/2}\leq c\beta^{2-q},
\end{split}
\end{align}
where $c=c(c_\mu)$. Assume now any $\delta>0$. By \eqref{higherint of h when centered in G} and by the definition of $G(\delta\beta)$, we have for $\nu$-almost every $z'\in G(\beta)$,
\begin{align*}
\vint_{Q_{5A\rho(z')}(z')} h^2\, d\nu &\leq c  \delta^2 \beta^2+\left(\frac{c}{\nu(Q_{A\rho(z')}(z'))} \int_{Q_{A\rho(z')}(z')\cap G(\delta \beta)}h^q\,d\nu \right)^{2/q}.
\end{align*}
By \eqref{local comparability of alpha} and \eqref{comparability of balls for g}, we can now choose a small enough positive number $\delta(c_\mu, A, \gamma)<1$ to absorb the first term on the right hand side into the left hand side. We obtain a positive $c=c(A,c_\mu,\gamma)$, such that for $\nu$-almost every $z'\in G(\beta)$ and any $\beta>\beta_0$, after using \eqref{q boundedness of h},
\begin{equation}\label{higherint2 of h}
\begin{split}
&\vint_{Q_{5A\rho(z')}(z')} h^2\, d\nu 
\leq  \beta^{2-q} \frac{c}{\nu(Q_{A \rho(z')}(z'))} \int_{Q_{A\rho(z')}(z')\cap G(\delta \beta)}h^q\,d\nu.
\end{split} 
\end{equation} 

The collection $\{\,Q_{A\rho(z')}(z')\,:\,z'\in G(\beta)\,\}$ is now an open cover of $G(\beta)$. By the Vitali covering lemma, there exists a countable and pairwise disjoint subcollection $\{\,Q_{A\rho(z_i')}(z_i')\,:\,z_i'\in G(\beta)\,\}_{i=1}^\infty$, such that 
\begin{align*}
G(\beta)\subset \bigcup_{i=1}^\infty Q_{5 A \rho(z_i')}(z_i')\subset Q_{2R}.
\end{align*}
The last inclusion follows from the fact that $5A \rho(z)\leq r(z)$. This property is the reason why we introduced the number $5$ into the proof earlier. Now we can write for any $\beta>\beta_0$, after multiplying inequality \eqref{higherint2 of h} with $\nu(Q_{A\rho(z')}(z'))$ and using the doubling property of $\mu$,
\begin{equation}\label{higherint2 of h at lower levels}
\begin{split}
\int_{G(\beta)}&h^2\,d\nu\leq \sum_{i=1}^\infty \int_{Q_{5A\rho(z_i')}(z_i')} h^2\, d\nu\\&\leq  \sum_{i=1}^\infty c \beta^{2-q}\int_{Q_{A\rho(z_i')}(z_i')\cap G(\delta \beta)}h^q\,d\nu\leq c\beta^{2-q}\int_{G(\delta \beta)} h^q\, d\nu.
\end{split}
\end{equation}
 From now on the higher integrability result is a consequence of \eqref{higherint2 of h at lower levels} and Fubini's theorem. To see this, we integrate
over $G(\beta_0)$ and use Fubini's theorem to obtain
\begin{equation*}
\begin{split}
\int_{G(\beta_0)}&h^{2+\eps} \, d \nu=
\int_{G(\beta_0)}\l(\int_{\beta_0}^{h} \eps \beta^{\eps-1} \,d
\beta\,+(\beta_0)^{\eps}\r)  h^2 \,d\nu
\\
&\hspace{-2 em}=
  \int_{\beta_0}^{\infty}
\eps \beta^{\eps-1}\int_{G(\beta)}h^2  \,d \nu \,d \beta\, +
(\beta_0)^{\eps} \int_{G(\beta_0)} h^2 \,d \nu,
\end{split}
\end{equation*}
and now by \eqref{higherint2 of h at lower levels}
\begin{align*}
 \int_{\beta_0}^{\infty}
 \eps\beta^{\eps-1}\int_{G(\beta)}h^2  \,d \nu \,d \beta\, \leq  c\int_{\beta_0}^{\infty} \varepsilon \beta^{\varepsilon+1-q} \int_{G(\delta \beta)}h^q\,d\nu\, d\beta.
\end{align*}
By Fubini's
theorem again, we see that
\begin{align*}
 &\int_{\beta_0}^{\infty} \eps
 \beta^{\eps+1-{q} } \int_{G(\delta \beta)} h^{{q}} \,d \nu\,d \beta\, +
\beta_0^{\eps} \int_{G(\beta_0)} h^2 \,d \nu
\\&\quad= 
 \eps\int_{G(\delta \beta_0)} \l(\int_{\beta_0}^{h/\delta}
 \beta^{\eps-1+2-{q} } \,d \beta\r) h^{{q}} \,d \nu\, +
\beta_0^{\eps} \int_{G(\beta_0)} h^2 \,d \nu\\&\quad\leq
\frac{\eps}{\delta^{2+\varepsilon-q}(\eps+2-{q}) } \int_{G( \beta_0)}
 h^{\eps+2} \,d \nu+
\beta_0^{\eps} \int_{G(\delta\beta_0)} h^2 \,d \nu,
\end{align*}
where $c=c(A,c_\mu,\gamma)$.  Observe that in the last step we also used the fact that $h^{\eps+2}\leq
  \beta_0^{\eps} h^2$  in $G(\delta
  \beta_0)\setminus G(\beta_0)$. We can now choose a positive $\eps=\eps(c_\mu, A, \gamma, q)$ small enough
to absorb the term containing $h^{2+\eps}$ into the left hand side of \eqref{higherint2 of h at lower levels}, and
conclude that
\begin{equation}
\label{melkein}
\begin{split}
&\int_{G(\beta_0)} h^{2+\eps} \,d \nu  \leq c(\beta_0)^{\eps}
\int_{G(\delta\beta_0)} h^2 \,d \nu,
\end{split}
\end{equation}
where $c=(c_\mu, A,\gamma)$. In case the term containing $h^{2+\varepsilon}$ is infinite, we replace $h$ by $h_k=\min\{h,k\}$ where $k>\beta$. Starting from \eqref{higherint2 of h at lower levels} we estimate that
\begin{equation}\label{higherint for cutoff}
\begin{split}
\int_{\{h_k> \beta \}} h_k^{2-q} \, d\zeta & \leq c \beta^{2-q} \int_{\{h_k> \delta \beta \}} d \zeta.
\end{split}
\end{equation} 
where $d\zeta= h^q \, d\nu$. Performing now as above the calculations involving Fubini's theorem yields
\begin{align*}
\int_{\{h_k> \beta_0 \}} &h_k^{2+\varepsilon-q} \, d\zeta \leq \varepsilon c \int_{\{h_k> \beta_0 \}} h_k^{2+\varepsilon-q} \, d\zeta +\beta_0^\varepsilon \int_{\{h_k> \delta \beta_0\}}h_k^{2-q}\, d\zeta.
\end{align*}
Now we can absorb the term containing $h_k^{2+\varepsilon-q}$ into the left hand side side, and finally let $k\rightarrow \infty$ to obtain \eqref{melkein}.

Finally, from the definitions of the parabolic distance and the parabolic cylinder, it is again straighforward to check that $Q_R\subset Q_{4 r(z)}(z)$ for every $z\in Q_R$. Hence, by the doubling property of the measure,
\begin{align*}\label{boundedness of alpha}
\alpha(z) \leq \frac{\nu(Q_{2R})}{\nu(Q_R)}\frac{\nu(Q_{4 r(z)}(z))}{\nu(Q_{\frac{r(z)}{5A}}(z))}\leq  c_1, \qquad \textrm{for every} \,z\in Q_R, 
\end{align*}
where $c_1=c_1(c_\mu,A)>0$. On the other hand, clearly $\alpha(z)\geq 1$ for every $z\in Q_{2R}$. Now \eqref{melkein} and the definition of $\beta_0$ imply that
\begin{align*}
\int_{Q_R} g^{2+\varepsilon} \, d\nu &\leq c_1^{\frac{2+\varepsilon}{2}}\left((\beta_0)^\eps\int_{Q_R\setminus G(\beta_0)} h^2 \, d\nu+\int_{G(\beta_0)} h^{2+\eps} \, d\nu \right)\\
&\leq c\frac{1}{(\nu(Q_{2R}))^{\varepsilon/2}}\left(\int_{Q_{2R}} g^2 \, d\nu \right)^{\frac{2+\eps}{2}},
\end{align*}
 where $c=c(c_\mu, A,\gamma)>0$. From this expression the proof can readily be completed.
\end{proof}
We conclude this paper by stating and proving the main result and a corollary.
\begin{theorem}[Local higher integrability]\label{local higher integrability}
Let $u\in L_\textrm{loc}^2(0,T;N_\textrm{loc}^{1,2}(\Omega))$ be a parabolic $K$-quasiminimizer.
Then there exists positive constants $\varepsilon=\eps(c_\mu, c_P,\lambda)$ and $c=c(c_\mu,c_P,\lambda,K)$, so that for every $z_0$ and $R$ such that $Q_{2R}(z_0)\subset \Omega_T$, we have
\[
\begin{split}
&\left(\vint_{Q_{R}(z)} g_u^{2+\eps} \,d \nu\right)^\frac{1}{2+\eps}\leq  c\l(\vint_{Q_{2R}(z)} g_u^2 \,d \nu\r)^{\frac{1}{2}}.
\end{split}
\]

\begin{proof} Let $z$ and $R$ be such that $Q_{2R}(z)\subset \Omega_T$.
By Lemma \ref{reverse hölder away from boundary} there exists  a constant $c=c(c_\mu,c_P,\lambda, K)$, and a $1<q<2$, such that
 \begin{align*}
 &\vint_{Q_{\rho}} g_u^2 \, d \nu\leq \varepsilon c\vint_{\q{2\lambda\rho}} g_u^2 \,d \nu+2\varepsilon^{-1}c\left(\vint_{Q_{2\lambda \rho}} g_u^q\,d\nu\right)^\frac{2}{q}.
 \end{align*}
Theorem \ref{global gehring} with $A=2\lambda$ now completes the proof.
\end{proof}
\end{theorem}

\begin{corollary}
Let  $u\in L_\textrm{loc}^2(0,T;N_\textrm{loc}^{1,2}(\Omega))$ be a parabolic $K$-quasiminimizer.  Then there exists a positive constant $\varepsilon=\eps(c_\mu, c_P,\lambda)$, such that for any compact $F\subset \Omega_T$, we have
\begin{align*}
\left(\int_{F} g_u^{2+\eps} \,d \nu\right)^\frac{1}{2+\eps}<\infty.
\end{align*}
\begin{proof}
Since $F$ is compact, there exists a finite collection of points $z_i\in F$ and parabolic  cylinders $Q_{R_i}(z_i)$ such that $Q_{2R_i}(z_i)\subset \Omega_T$ and
\begin{align*}
F\subset \bigcup_{i=1}^n Q_{R_i}(z_i).
\end{align*}
By Theorem \ref{local higher integrability}, we have
\begin{align*}
&\int_{F} g_u^{2+\eps} \,d \nu\leq \max_{1\leq i \leq n} \nu(Q_{R_i}(z_i))\sum_{i=1}^n\vint_{Q_{R_i}(z_i)} g_u^{2+\eps} \,d \nu\\
&\leq c\max_{1\leq i \leq n} \nu(Q_{R_i}(z_i)) \sum_{i=1}^n \l(\vint_{Q_{2R_i}(z_i)} g_u^2 \,d \nu\r)^{\frac{2+\varepsilon}{2}}<\infty.
\end{align*}
\end{proof}
\end{corollary}

\def\cprime{$'$} \def\cprime{$'$}

\end{document}